\documentclass[11pt]{article}

\usepackage[margin=1in]{geometry}
\usepackage{amsmath,amssymb,amsthm,mathtools}
\usepackage{array,booktabs,tabularx}
\usepackage{enumitem}
\usepackage{float}
\usepackage[normalem]{ulem}
\usepackage{tikz}
\usetikzlibrary{arrows.meta}
\usepackage{pgfplots}
\pgfplotsset{compat=1.18}
\usepgfplotslibrary{fillbetween}
\usepackage[colorlinks=true,linkcolor=blue,citecolor=blue,urlcolor=blue]{hyperref}

\newtheorem{theorem}{Theorem}[section]
\newtheorem{proposition}[theorem]{Proposition}
\newtheorem{lemma}[theorem]{Lemma}
\newtheorem{corollary}[theorem]{Corollary}

\newtheorem{example}[theorem]{Example}
\theoremstyle{remark}
\newtheorem{remark}[theorem]{Remark}

\newcommand{\R}{\mathbb R}
\newcommand{\Lcone}{\mathcal L}

\newcommand{\rank}{\operatorname{rank}}
\newcommand{\Span}{\operatorname{span}}
\newcolumntype{Y}{>{\raggedright\arraybackslash}X}

\title{An Exact Dual for Second-Order Cone Programming\\
Using Only Lorentz-Cone Constraints}
\author{Hao Hu\\
\small School of Mathematical and Statistical Sciences, Clemson University\\
\small Clemson, South Carolina 29634, USA\\
\small \href{mailto:hhu2@clemson.edu}{hhu2@clemson.edu}
\quad\textperiodcentered\quad
\url{https://huhao.org/}}
\date{September 6, 2026}

\begin{document}
\maketitle

\begin{abstract}
Without a constraint qualification, second-order cone programs can have a
duality gap or an unattained dual value.  For every such program with
arbitrary real coefficients, we construct an exact dual using only affine
equations and memberships in products of Lorentz cones.  Every feasible dual
point gives a valid bound.  Whenever the primal is feasible with finite value,
the dual attains that value.  This answers a question raised by
P\'olik and Terlaky and reiterated in subsequent work: whether an exact dual
for second-order cone programming can use only Lorentz-cone constraints.  The
same construction produces an affine second-order cone system that is
feasible exactly when a given affine
second-order cone system is infeasible, including in the weakly infeasible
case.  Both formulations have polynomial size and are obtained uniformly from
the input coefficients using only fixed rational constants.  These are
formulation-size results over exact real data; they do not imply
polynomial-time solvability or polynomial bounds on the bit length of rational
certificates.
\end{abstract}



\section{Introduction}

Strong duality is central to the interpretation and use of conic
optimization: dual-feasible points provide bounds, and dual attainment
provides optimality certificates.  Under a constraint qualification such as
Slater's condition, second-order cone programming (SOCP) has the familiar
strong-duality theory.  Without such a condition, however, the ordinary conic
dual may have a duality gap or may fail to attain its value.

Facial reduction, introduced by Borwein and Wolkowicz
\cite{BorweinWolkowicz1981Facial,BorweinWolkowicz1981Regularizing}, restores
strong duality by replacing the original cone with the minimal face that
contains the feasible set.  This procedure is sequential.  To turn it into
one explicit exact dual, the required facial-reduction information must
instead be encoded in a single finite conic formulation.  For semidefinite
programming, Ramana's dual accomplishes this using semidefinite constraints
\cite{Ramana1997}.  For SOCP, the corresponding challenge is to obtain an
exact dual using only Lorentz-cone constraints, without passing to a larger
cone class.  The exact dual constructed here answers a question raised by
P\'olik and Terlaky in 2007: whether such a dual can be expressed using only
Lorentz cones \cite{PolikTerlaky2007}.

Weakly infeasible systems present a related certification problem: their
distance to feasibility is zero, so ordinary strict separation does not yield
an attained normalized infeasibility certificate.  The question addressed
here is whether the facial-reduction information and terminal dual slack can
be encoded together in one polynomial-size SOCP, so that exact duality and
infeasibility certification remain within the original cone class.

\subsection{Problem setting}

Let \(K\subseteq\R^n\) be a product of Lorentz cones and nonnegative scalar
cones.  We study the conic program
\begin{equation}
 (P)\qquad
 p^*=\inf\{\langle c,x\rangle:Ax=b,\ x\in K\},
 \label{eq:intro-primal}
\end{equation}
where \(A\in\R^{m\times n}\), \(b\in\R^m\), and \(c\in\R^n\).  A
\emph{SOCP formulation} uses only affine equations and memberships in
products of Lorentz cones; a nonnegative scalar cone can be represented as a
fixed Lorentz-cone slice.  An exact dual for \eqref{eq:intro-primal} must
satisfy weak duality for every dual-feasible point and attain \(p^*\) whenever
the primal is feasible with finite value, without assuming a constraint
qualification.

For an affine SOCP feasibility instance \(\Sigma\), we seek an explicit
transformation \(\mathfrak N\) into another affine SOCP feasibility system
such that
\begin{equation}
 \mathfrak N(\Sigma)\text{ is feasible}
 \quad\Longleftrightarrow\quad
 \Sigma\text{ is infeasible}.
 \label{eq:intro-infeasibility-system}
\end{equation}
A feasible point of \(\mathfrak N(\Sigma)\) is an existential conic
certificate of infeasibility; the construction does not assert that such a
point can be found efficiently.

\subsection{Main results and implications}

The construction separates Pataki's general principle from the SOCP-specific
contribution.  Pataki proved that an exact extended dual can use a set-valued
mapping in place of the tangent-space constraint, provided that the mapping
is contained in the true tangent space and recovers every tangent vector after
positive rescaling of the vector at which the tangent space is taken
\cite[Corollary~3]{Pataki2013}.  The contribution specific to SOCP is an
affine representation of such a mapping using only Lorentz and nonnegative
scalar cones.  The argument has three parts.
\begin{enumerate}[leftmargin=2.2em]
\item Subsections~\ref{sec:lorentz-tangent-spaces}--
\ref{sec:lifted-tangent-relation} identify fixed linear maps that generate
the tangent spaces of scalar and Lorentz cones, and use them to construct an
SOCP-representable set-valued mapping satisfying these containment and
rescaling requirements.
\item Subsection~\ref{sec:facial-reduction-encoding} combines this mapping
with a backward-rescaling argument and expresses the resulting
facial-reduction data using explicit product-cone variables.
\item Subsections~\ref{sec:direct-exact-dual} and
\ref{sec:infeasibility-certificates} derive the polynomial-size exact dual in
SOCP form and apply standard homogenization to obtain an infeasibility
certificate within the same cone class.
\end{enumerate}

Theorem~\ref{thm:direct-exact-dual} constructs an exact dual in SOCP form for
\eqref{eq:intro-primal}.  Every feasible point gives a valid lower bound, and
the dual attains the primal value whenever that value is finite and the primal
is feasible.  Corollary~\ref{cor:pure-socp-infeasibility-certificate}
specializes the same formulation to an SOCP system satisfying
\eqref{eq:intro-infeasibility-system}, including for weakly infeasible
systems.
Both formulations have polynomial size and apply to arbitrary real
coefficients.

The result keeps exact duality and infeasibility certification within the
original cone class.  It removes the duality gap and nonattainment at the
level of the extended formulation, but it does not imply that the formulation
is numerically well conditioned.  The construction uses a set-valued mapping
contained in the tangent space and recovers the required tangent directions
by rescaling along one facial-reduction sequence, rather than representing
the full tangent-space constraint.  This restricted existential
representation is the source of the polynomial size.

These are formulation-size results over exact real data; they do not imply
polynomial solution time or polynomial bit length of rational certificates.

\subsection{Relation to prior work}

Ramana's exact dual for semidefinite programming (SDP) records enough of a
facial-reduction sequence to recover strong duality \cite{Ramana1997}.
Ramana, Tun\c{c}el, and Wolkowicz made the connection between this extended
dual and facial reduction explicit \cite{RamanaTuncelWolkowicz1997}.  P\'olik
and Terlaky developed exact duality over symmetric cones, but their
Lorentz-cone specialization uses Lorentz--Siegel cones rather than a
formulation confined to Lorentz cones.  They explicitly left open whether an
exact dual using only Lorentz cones exists \cite[p.~19]{PolikTerlaky2007}.
As late as 2021, Louren\c{c}o, Muramatsu, and Tsuchiya likewise observed that
it was unclear whether the alternative dual obtained from Pataki's approach
could be expressed using second-order cone constraints
\cite[p.~461]{LourencoMuramatsuTsuchiya2021}.

Pataki expressed exact duality through facial reduction and tangent spaces.
More specifically, his Corollary~3 permits the tangent-space constraint to be
replaced by any set-valued mapping satisfying the containment and
positive-rescaling properties used here \cite[Corollary~3]{Pataki2013}.  That
result is an abstract exact-duality principle; it does not provide an affine
formulation of such a mapping using Lorentz-cone constraints.  For broader
background on degeneracy and facial reduction in conic optimization, see
Drusvyatskiy and Wolkowicz \cite{DrusvyatskiyWolkowicz2017}.
Section~\ref{sec:related} gives the precise correspondence with Pataki's
framework and compares the present formulation with other exact-duality
certificates.

\subsection{Organization}

Section~\ref{sec:preliminaries} reviews the required cone geometry and facial
reduction.  Section~\ref{sec:lifted-tangent-exact-dual} constructs the
SOCP-representable mapping, applies the backward-rescaling argument to
represent facial-reduction sequences, and derives the exact dual in SOCP form
and the infeasibility certificate.  It concludes with a weakly infeasible
example.  Section~\ref{sec:related} gives the precise correspondence with
the extended-duality framework above, compares the construction with other exact-duality
certificates, and states its scope.

\section{Conic preliminaries and facial reduction}
\label{sec:preliminaries}

\subsection{Scalar and Lorentz cones}

For \(n\geq2\), the Lorentz cone is
\[
 \Lcone_n=\{(t,z)\in\R\times\R^{n-1}:t\geq\|z\|_2\}.
\]
It is closed, pointed, and self-dual.  The Jordan ranks of \(\Lcone_n\) and
\(\R_+\) are two and one, respectively, and rank is additive across products:
\[
 \rank\!\left(\prod_i Q_i\right)=\sum_i\rank Q_i.
\]
For \(\Lcone_n\), rank two is also visible from its face structure: its only
nonzero proper faces are rays, so every strictly descending face chain has at
most two proper steps.  Scalar factors \(\R_+\) may instead be represented by
the Lorentz slice \((s,0)\in\Lcone_2\).  Free scalars may be written as
differences of nonnegative scalars.

\subsection{Facial reduction for cone--subspace intersections}

For a face \(F\) of \(K\), write
\[
 F^*=\{y:\langle y,x\rangle\geq0\text{ for every }x\in F\},
 \qquad
 F^\perp=(\Span F)^\perp.
\]
Let \(L\subseteq\mathbb R^n\) be a linear subspace and fix \(\ell\geq0\).
A facial-reduction sequence of length \(\ell\) for \(L\cap K\) is specified
by reducing directions
\((g_1,\ldots,g_\ell)\) and the sequence of faces
\((F_0,\ldots,F_\ell)\) generated recursively from \(F_0=K\) by
\begin{equation}
 g_j\in L^\perp\cap F_{j-1}^*,\qquad
 F_j=F_{j-1}\cap g_j^\perp,\qquad
 F_j\subsetneq F_{j-1}
 \qquad(1\leq j\leq\ell).
 \label{eq:fr-chain-intro}
\end{equation}
Every point of \(L\cap K\) lies in every \(F_j\).  Each step strictly
decreases the Jordan rank, so \(\ell\leq\rank K\).  In particular, a
facial-reduction sequence reaching the minimal face has length at most
\(\rank K\).  This is the product-Lorentz specialization of the standard
bound by the longest chain of faces
\cite[Theorem~2.3]{PolikTerlaky2007}.

Scalar and Lorentz cones are nice \cite[p.~9]{Pataki2013}, and niceness is
preserved by finite products.  Hence \(K\) is self-dual and nice.  Therefore,
for every face \(F\) of \(K\),
\begin{equation}
 F^*=K+F^\perp.
 \label{eq:niceness}
\end{equation}

\section{The exact dual and infeasibility certificates}
\label{sec:lifted-tangent-exact-dual}

Pataki's simplified extended-dual framework permits the tangent-space
constraint to be replaced by a set-valued mapping that satisfies containment
and positive-rescaling conditions \cite[Corollary~3]{Pataki2013}.  The main task
in this section is to construct such a mapping using only the original
Lorentz and scalar cone blocks.  The construction separates two operations.  At
facial-reduction step \(j\), it represents the cumulative reducing vector as
\(y_j=q_j+z_j\), where \(q_j\) is a cone vector and \(z_j\) is tangent at the
preceding cone vector \(q_{j-1}\).  The final dual slack uses the same
decomposition but need not expose another face.  An affine lift represents
each tangent component using original cone blocks and fixed linear maps.  The
decomposition expresses the facial geometry, while the lift keeps the
formulation within SOCP.  Rescaling permits each indexed constraint to depend
only on its immediate predecessor.

We first construct the mapping for one cone block and introduce
product indices only when assembling the blocks.  Let \(Q\) be either
\(\mathbb R_+\) or a Lorentz cone \(\mathcal L_n\).  In the Lorentz case, use
coordinates \(0,\ldots,n-1\) and set
\[
 J=\operatorname{diag}(1,-1,\ldots,-1),\qquad
 S_{ab}=e_ae_b^T-e_be_a^T.
\]
Here \(e_a\) is the corresponding coordinate vector.  Associate with \(Q\)
the fixed map family
\begin{equation}
 \mathcal R=
 \begin{cases}
  \{I_1\},&Q=\mathbb R_+,\\[1mm]
  \{I_n\}\cup
  \{S_{ab}J:0\leq a<b<n\},&Q=\mathcal L_n.
 \end{cases}
 \label{eq:fixed-map-family}
\end{equation}

\subsection{Tangent spaces of Lorentz cones}
\label{sec:lorentz-tangent-spaces}

For \(Q=\mathcal L_n\) and \(u\in Q\), following Pataki, define the tangent
space
\begin{equation}
 \operatorname{tan}(u,Q):=(Q\cap u^\perp)^\perp.
 \label{eq:block-tangent-space}
\end{equation}
This is the specialization of \cite[equation~(2.2)]{Pataki2013} to the
self-dual cone \(Q\).

The following elementary identity motivates the fixed maps.  For every
nonzero \(w\in\R^n\),
\begin{equation}
 w^\perp=\operatorname{span}\{S_{ab}w:0\leq a<b<n\}.
 \label{eq:elementary-skew-span}
\end{equation}
The inclusion from right to left follows from skew symmetry.  Conversely, for
\(y\perp w\), the matrix
\[
 \Omega_y=\frac{yw^T-wy^T}{\|w\|_2^2}
\]
is skew-symmetric and satisfies \(\Omega_yw=y\).  Since the \(S_{ab}\) form a
basis of the skew-symmetric matrices, \(y\) belongs to the span in
\eqref{eq:elementary-skew-span}.

If \(v\) is a nonzero boundary point of \(\Lcone_n\), then
\(\Lcone_n\cap v^\perp=\R_+Jv\), and
\eqref{eq:elementary-skew-span} with \(w=Jv\) describes the orthogonal
complement of this ray.  This explains the maps \(S_{ab}J\) in
\eqref{eq:fixed-map-family}; the identity map supplies the remaining
direction when \(v\) is interior.

\begin{lemma}[Tangent-space generators]
\label{lem:orthogonality-transport}
Let \(Q=\mathcal L_n\), and let \(\mathcal R\) be the corresponding map
family in \eqref{eq:fixed-map-family}.  Then, for every \(v\in Q\),
\begin{equation}
 \operatorname{tan}(v,Q)=\operatorname{span}\{Rv:R\in\mathcal R\}.
 \label{eq:single-vector-tangent-span}
\end{equation}
In particular, \(Rv\in\operatorname{tan}(v,Q)\) for every
\(R\in\mathcal R\).  Equivalently, every
\(z\in\operatorname{tan}(v,Q)\) can be written as
\(z=\sum_{R\in\mathcal R}\alpha_RRv\) for some \(\alpha_R\in\mathbb R\).
\end{lemma}

\begin{proof}
If \(v=0\), both sides of
\eqref{eq:single-vector-tangent-span} are zero.  If \(v\) is a nonzero
boundary point, then \(Q\cap v^\perp=\mathbb R_+Jv\), and
\eqref{eq:elementary-skew-span} with \(w=Jv\) gives
\(\operatorname{tan}(v,Q)=(Jv)^\perp
=\operatorname{span}\{S_{ab}Jv:a<b\}\).  The identity generator adds no new
direction: writing \(v=(v_0,\bar v)\), the boundary relation
\(v_0=\|\bar v\|_2\) gives
\(\langle v,Jv\rangle=v_0^2-\|\bar v\|_2^2=0\), and hence
\(v\in(Jv)^\perp\).  If \(v\) is interior, then
\(\operatorname{tan}(v,Q)=\mathbb R^n\).  Since \(Jv\neq0\),
\eqref{eq:elementary-skew-span} shows that the skew images span the
\((n-1)\)-dimensional hyperplane \((Jv)^\perp\).  The identity image
\(v=I_nv\) lies outside this hyperplane because \(\langle v,Jv\rangle>0\);
together they span \(\mathbb R^n\).
\end{proof}

\paragraph{Scalar case.}
For \(Q=\mathbb R_+\), the same definition gives
\begin{equation}
 \operatorname{tan}(0,\mathbb R_+)=\{0\},\qquad
 \operatorname{tan}(u,\mathbb R_+)=\mathbb R\quad(u>0).
 \label{eq:scalar-tangent-space}
\end{equation}
Consequently, Lemma~\ref{lem:orthogonality-transport} remains valid for the
scalar cone with \(\mathcal R=\{I_1\}\).

\subsection{An SOCP-representable auxiliary mapping}
\label{sec:lifted-tangent-relation}

Pataki's extended-duality result \cite[Corollary~3]{Pataki2013} reduces the
exact-duality question to finding a set-valued mapping with the two properties
stated in Proposition~\ref{prop:lifted-tangent}.  The following definition
provides such a mapping for a scalar or Lorentz block; this is the central
SOCP-specific construction.

For the single block \(Q\) and map family \(\mathcal R\) above, define the
set-valued mapping
\begin{equation}
\begin{aligned}
 \operatorname{tan}'(u,Q):=\biggl\{z:\ &u=w+\sum_{R\in\mathcal R}(v_R^++v_R^-),\\
                         &z=\sum_{R\in\mathcal R}R(v_R^+-v_R^-),
 \quad w,v_R^+,v_R^-\in Q\biggr\}.
\end{aligned}
 \label{eq:lifted-block-tangent}
\end{equation}
The first line allocates conic pieces of \(u\); the second applies fixed maps
and permits arbitrary signs.  Thus \eqref{eq:lifted-block-tangent} is an
affine lifted representation over copies of \(Q\).

\begin{proposition}[Containment and rescaling]
\label{prop:lifted-tangent}
For \(Q\) and \(\mathcal R\) as in \eqref{eq:fixed-map-family}, and every
\(u\in Q\),
\begin{enumerate}[label=\textup{(\alph*)},leftmargin=2.2em]
\item \(\operatorname{tan}'(u,Q)\subseteq\operatorname{tan}(u,Q)\).
\item If \(z\in\operatorname{tan}(u,Q)\) and \(\gamma>0\), then
\(\gamma z\in\operatorname{tan}'(Mu,Q)\) for some \(M>0\).
\end{enumerate}
\end{proposition}

\begin{proof}
For part (a), take \(z\in\operatorname{tan}'(u,Q)\) and
\(h\in Q\cap u^\perp\).  The first line of
\eqref{eq:lifted-block-tangent} gives
\[
 0=\langle u,h\rangle
  =\langle w,h\rangle
   +\sum_{R\in\mathcal R}
     \bigl(\langle v_R^+,h\rangle+\langle v_R^-,h\rangle\bigr).
\]
Every term on the right is nonnegative because all displayed vectors lie in
the self-dual cone \(Q\).  Hence
\(\langle v_R^+,h\rangle=\langle v_R^-,h\rangle=0\) for every
\(R\in\mathcal R\).  Lemma~\ref{lem:orthogonality-transport}, together with
the scalar case in
\eqref{eq:scalar-tangent-space}, gives
\[
 Rv_R^+\in\operatorname{tan}(v_R^+,Q)
       =\bigl(Q\cap(v_R^+)^\perp\bigr)^\perp.
\]
Since \(h\in Q\cap(v_R^+)^\perp\), the definition of the orthogonal
complement gives \(\langle Rv_R^+,h\rangle=0\).  Similarly,
\(h\in Q\cap(v_R^-)^\perp\), and the same lemma gives
\(\langle Rv_R^-,h\rangle=0\).  Taking the inner product of the second line
of \eqref{eq:lifted-block-tangent} with \(h\) now gives
\[
 \langle z,h\rangle
 =\sum_{R\in\mathcal R}
   \bigl(\langle Rv_R^+,h\rangle-\langle Rv_R^-,h\rangle\bigr)
 =0.
\]

For part (b), Lemma~\ref{lem:orthogonality-transport}, together with
\eqref{eq:scalar-tangent-space}, gives
coefficients \(\alpha_R\in\mathbb R\) such that
\(z=\sum_R\alpha_RRu\).  Write
\((\alpha)_+:=\max\{\alpha,0\}\) and
\((\alpha)_-:=\max\{-\alpha,0\}\), so that
\(\alpha=(\alpha)_+-(\alpha)_-\) and
\(|\alpha|=(\alpha)_++(\alpha)_-\).  Choose
\(M\geq\gamma\sum_R|\alpha_R|\), with \(M>0\), and set
\[
 v_R^+=\gamma(\alpha_R)_+u,\qquad
 v_R^-=\gamma(\alpha_R)_-u,\qquad
 w=\left(M-\gamma\sum_R|\alpha_R|\right)u.
\]
Since \(\gamma>0\) and both \((\alpha_R)_+\) and \((\alpha_R)_-\) are
nonnegative, the vectors \(v_R^+\) and \(v_R^-\) lie in \(Q\).  The choice
of \(M\) makes the coefficient of \(u\) in \(w\) nonnegative, so
\(w\in Q\).  Moreover,
\[
\begin{aligned}
 w+\sum_R(v_R^++v_R^-)
 &=\left(M-\gamma\sum_R|\alpha_R|\right)u
   +\gamma\sum_R\bigl((\alpha_R)_++(\alpha_R)_-\bigr)u
  =Mu,\\
 \sum_R R(v_R^+-v_R^-)
 &=\gamma\sum_R\bigl((\alpha_R)_+-(\alpha_R)_-\bigr)Ru
  =\gamma\sum_R\alpha_RRu
  =\gamma z.
\end{aligned}
\]
Thus the defining equations in \eqref{eq:lifted-block-tangent} hold with
\(u\) replaced by \(Mu\) and \(z\) replaced by \(\gamma z\).  Hence
\(\gamma z\in\operatorname{tan}'(Mu,Q)\).
\end{proof}

\paragraph{Product assembly.}
Return to the product cone in \eqref{eq:intro-primal} and write
\[
 K=K_1\times\cdots\times K_\nu,
\]
where every block is either \(\mathbb R_+\) or a Lorentz cone
\(\mathcal L_{n_i}\), and put \(n_i=1\) for a scalar block.  Apply
\eqref{eq:fixed-map-family} in the coordinates of each block \(K_i\), and
denote the resulting fixed map family by \(\mathcal R_i\).

For \(u=(u_1,\ldots,u_\nu)\in K\), set
\begin{equation}
 \operatorname{tan}(u,K):=\prod_{i=1}^\nu\operatorname{tan}(u_i,K_i),\qquad
 \operatorname{tan}'(u,K):=\prod_{i=1}^\nu\operatorname{tan}'(u_i,K_i).
 \label{eq:product-tangent-relations}
\end{equation}
Because each \(K_i\) is self-dual,
\(\langle u_i,x_i\rangle\geq0\) for \(u_i,x_i\in K_i\).  Thus
\(x\in K\cap u^\perp\) if and only if
\(x_i\in K_i\cap u_i^\perp\) for every \(i\), so
\(\operatorname{tan}(u,K)=(K\cap u^\perp)^\perp\).

The product mapping has the following explicit blockwise SOCP representation.  For
\(z=(z_1,\ldots,z_\nu)\), the membership
\(z\in\operatorname{tan}'(u,K)\) holds if and only if, for every block
\(i\), there are \(w_i,v_{iR}^+,v_{iR}^-\in K_i\) satisfying
\begin{align}
 u_i
  &=w_i+\sum_{R\in\mathcal R_i}(v_{iR}^++v_{iR}^-),
 \label{eq:product-lift-source}\\
 z_i
  &=\sum_{R\in\mathcal R_i}R(v_{iR}^+-v_{iR}^-).
 \label{eq:product-lift-image}
\end{align}
These relations are affine, and every conic variable belongs to an original
cone block.

\begin{corollary}[Product containment and rescaling]
\label{cor:product-lifted-tangent}
For every \(u\in K\),
\begin{enumerate}[label=\textup{(\alph*)},leftmargin=2.2em]
\item \(\operatorname{tan}'(u,K)\subseteq\operatorname{tan}(u,K)\).
\item If \(z\in\operatorname{tan}(u,K)\) and \(\gamma>0\), then
\(\gamma z\in\operatorname{tan}'(Mu,K)\) for some \(M>0\).
\end{enumerate}
\end{corollary}

\begin{proof}
Part (a) follows by applying Proposition~\ref{prop:lifted-tangent}(a) in
each block.  For part (b), apply Proposition~\ref{prop:lifted-tangent}(b)
in block \(i\) to obtain \(M_i>0\), and set
\(M=\max_{1\leq i\leq\nu}M_i\).  Define
\(w_i':=w_i+(M-M_i)u_i\).  Since \(M\geq M_i\), we have
\(w_i'\in K_i\).  Replacing \(w_i\) by \(w_i'\) changes the left-hand side
of \eqref{eq:product-lift-source} from \(M_iu_i\) to \(Mu_i\), while
\eqref{eq:product-lift-image} remains unchanged because it does not involve
\(w_i\).
\end{proof}

\subsection{An SOCP representation of facial-reduction sequences}
\label{sec:facial-reduction-encoding}

This subsection specializes the sequence-encoding and backward-rescaling
argument in Pataki's Corollary~3 to the mapping constructed above
\cite[Corollary~3 and equation~(5.32)]{Pataki2013}.  We include the argument
to fix the correspondence with the explicit SOCP variables used in the final
formulation.

We first derive two elementary cumulative identities for a facial-reduction
sequence that will be used in the SOCP representation.  Let
\(g_1,\ldots,g_\ell\) be vectors and \(F_0,\ldots,F_\ell\) be faces, with
\(F_0=K\), satisfying
\[
 g_j\in F_{j-1}^*,\qquad
 F_j=F_{j-1}\cap g_j^\perp
 \qquad(1\leq j\leq\ell).
\]
These relations hold for every facial-reduction sequence and remain valid if
zero directions are appended after the sequence reaches the minimal face.
By \eqref{eq:niceness}, choose
\begin{equation}
 g_j=a_j+r_j,\qquad
 a_j\in K,\qquad r_j\in F_{j-1}^\perp
 \qquad(1\leq j\leq\ell).
 \label{eq:fr-direction-decomposition}
\end{equation}
Set \(\bar a_0=\bar r_0=0\) and, for \(1\leq j\leq\ell\), define
\begin{equation}
 \bar a_j:=\sum_{k=1}^j a_k,
 \qquad
 \bar r_j:=\sum_{k=1}^j r_k.
 \label{eq:cumulative-parts}
\end{equation}
Then
\begin{equation}
 F_j=K\cap\bar a_j^\perp\quad(0\leq j\leq\ell),
 \qquad
 \bar r_j\in F_{j-1}^\perp\quad(1\leq j\leq\ell).
 \label{eq:cumulative-facial-reduction}
\end{equation}
Indeed, \(r_j\perp F_{j-1}\) gives
\(F_j=F_{j-1}\cap a_j^\perp\).  Iteration yields
\(F_j=K\cap a_1^\perp\cap\cdots\cap a_j^\perp\).  For \(x\in K\),
self-duality gives \(\langle a_k,x\rangle\geq0\) for every \(k\).  Hence
\(\langle\bar a_j,x\rangle=0\) if and only if
\(\langle a_k,x\rangle=0\) for every \(1\leq k\leq j\).  This proves the
first identity, \(F_j=K\cap\bar a_j^\perp\), in
\eqref{eq:cumulative-facial-reduction}.  Moreover, for \(k\leq j\), the inclusion
\(F_{j-1}\subseteq F_{k-1}\) implies
\(F_{k-1}^\perp\subseteq F_{j-1}^\perp\).  Hence
\[
 r_k\in F_{k-1}^\perp\subseteq F_{j-1}^\perp
 \qquad(k\leq j).
\]
Since \(F_{j-1}^\perp\) is a linear subspace, it follows that
\(\bar r_j=\sum_{k=1}^j r_k\in F_{j-1}^\perp\).  This proves the second
identity in \eqref{eq:cumulative-facial-reduction}.

Fix \(s\geq1\), apply the preceding identities with \(\ell=s-1\), and let
\(g_s\in F_{s-1}^*\) denote a terminal vector, which will be the dual slack
in the application below.  We next encode these data through
\(\operatorname{tan}'\); equations~\eqref{eq:product-lift-source}--
\eqref{eq:product-lift-image} will then express the encoding as an SOCP.  For
each \(1\leq j\leq s\), introduce
\(y_j\in\mathbb R^n\), \(q_j\in K\), and \(z_j\in\mathbb R^n\).  Set
\(q_0=0\) and \(z_1=0\), and impose
\begin{equation}
 y_j=q_j+z_j\quad(1\leq j\leq s),
 \qquad z_j\in\operatorname{tan}'(q_{j-1},K)\quad(2\leq j\leq s).
 \label{eq:adjacent-chain-relation}
\end{equation}
\begin{samepage}
When \eqref{eq:adjacent-chain-relation} represents a facial-reduction
sequence, the three vectors have the following roles for \(1\leq j<s\).
\begin{enumerate}[leftmargin=2.2em]
\item \(y_j\in F_{j-1}^*\) is the reducing direction used by the represented
sequence.
\item \(q_j\in K\) is its cone representative.
\item \(z_j\in F_{j-1}^\perp\) is the correction satisfying
\(y_j=q_j+z_j\).
\end{enumerate}
\end{samepage}
Since \(z_j\perp F_{j-1}\), the vectors \(y_j\) and \(q_j\) agree on
\(F_{j-1}\) and expose the same next face.  More precisely,
Lemma~\ref{lem:adjacent-lifted-tangent-chain}(b) shows that the representation
can be chosen, for suitable scalars \(M_1,\ldots,M_{s-1}>0\), so that
\[
 y_j=M_j\sum_{k=1}^j g_k,\qquad
 q_j=M_j\bar a_j,\qquad z_j=M_j\bar r_j
 \qquad(1\leq j<s).
\]
Thus \(y_j\) is generally not the original direction \(g_j\), but a positive
multiple of the cumulative reducing direction through step \(j\).  At
\(j=s\), the same decomposition represents the terminal dual slack:
\(y_s=g_s\), \(q_s=a_s\), and \(z_s=r_s\).  Table~\ref{tab:facial-chain-encoding}
summarizes this correspondence.

\begin{table}[htbp]
\centering
\small
\begin{tabularx}{\textwidth}{@{}lYY@{}}
\toprule
Role & Ordinary facial-reduction data & SOCP representation \\
\midrule
Cumulative reducing vector, \(j<s\)
 & \(\sum_{k=1}^j g_k=\bar a_j+\bar r_j\)
 & \(y_j=q_j+z_j=M_j\sum_{k=1}^j g_k\) \\
Cumulative cone part, \(j<s\)
 & \(\bar a_j=\sum_{k=1}^j a_k\)
 & \(q_j=M_j\bar a_j\) \\
Cumulative tangent part, \(j<s\)
 & \(\bar r_j=\sum_{k=1}^j r_k\)
 & \(z_j=M_j\bar r_j\) \\
Face after step \(j<s\)
 & \(F_j\)
 & \(F_j=K\cap q_j^\perp\) \\
Terminal dual slack
 & \(g_s=a_s+r_s\in F_{s-1}^*\)
 & \(y_s=g_s,\ q_s=a_s,\ z_s=r_s\) \\
\bottomrule
\end{tabularx}
\caption{The ordinary facial-reduction data and their SOCP representation in
Lemma~\ref{lem:adjacent-lifted-tangent-chain}(b).}
\label{tab:facial-chain-encoding}
\end{table}

Expand each condition \(z_j\in\operatorname{tan}'(q_{j-1},K)\) using the
blockwise constraints
\eqref{eq:product-lift-source}--\eqref{eq:product-lift-image},
using, for each \(2\leq j\leq s\), separate auxiliary variables
\[
 w_{ji},v_{jiR}^+,v_{jiR}^-\in K_i.
\]
The
resulting system contains only affine equations and memberships in the
original cone blocks.

The next lemma is a self-contained reformulation, in the present notation, of
the two consequences of Pataki's argument needed below.  Part~(a) establishes
the nonnegativity implication needed for weak duality.  Part~(b) shows that the
constraints represent the directions in a facial-reduction sequence,
including after zero directions are appended.

\begin{lemma}[SOCP representation of facial-reduction sequences]
\label{lem:adjacent-lifted-tangent-chain}
The variables above have the following properties.
\begin{enumerate}[label=\textup{(\alph*)},leftmargin=2.2em]
\item If \(x\in K\) and \(\langle y_j,x\rangle=0\) for
\(1\leq j<s\), then
\[
 \langle q_j,x\rangle=0\quad(1\leq j<s),
 \qquad \langle y_s,x\rangle\geq0.
\]
\item Let \(F_0=K\), and suppose that, for \(1\leq j<s\),
\[
 g_j\in F_{j-1}^*,\qquad
 F_j=F_{j-1}\cap g_j^\perp.
\]
For every \(g_s\in F_{s-1}^*\), the lifted variables can be
chosen so that, for some \(M_1,\ldots,M_{s-1}>0\),
\begin{equation}
 y_j=M_j\sum_{k=1}^j g_k\quad(1\leq j<s),
 \qquad y_s=g_s.
 \label{eq:facial-chain-encoding}
\end{equation}
\end{enumerate}
\end{lemma}

\begin{proof}
For part (a), we prove by induction that
\(\langle z_j,x\rangle=0\) for \(1\leq j\leq s\).  The initial case follows
from \(z_1=0\).  Suppose that \(\langle z_j,x\rangle=0\) for some \(j<s\).
Since \(\langle y_j,x\rangle=0\) and \(y_j=q_j+z_j\), we obtain
\(\langle q_j,x\rangle=0\), and hence \(x\in K\cap q_j^\perp\).  On the
other hand, \eqref{eq:adjacent-chain-relation} and Corollary
\ref{cor:product-lifted-tangent}(a) give
\[
 z_{j+1}\in\operatorname{tan}'(q_j,K)
 \subseteq(K\cap q_j^\perp)^\perp.
\]
Therefore \(\langle z_{j+1},x\rangle=0\), completing the induction.  The
argument also proves \(\langle q_j,x\rangle=0\) for every \(j<s\).  Finally,
\(\langle z_s,x\rangle=0\), and therefore
\(\langle y_s,x\rangle=\langle q_s,x\rangle\geq0\).

For part (b), use \eqref{eq:fr-direction-decomposition} to decompose
\(g_1,\ldots,g_{s-1}\).  Separately, since \(g_s\in F_{s-1}^*\),
\eqref{eq:niceness} gives
\[
 g_s=a_s+r_s,\qquad a_s\in K,\qquad r_s\in F_{s-1}^\perp.
\]
Define \(\bar a_j\) and \(\bar r_j\) by
\eqref{eq:cumulative-parts}.  Applying
\eqref{eq:cumulative-facial-reduction} to
\(g_1,\ldots,g_{s-1}\) gives
\(F_j=K\cap\bar a_j^\perp\) and
\(\bar r_j\in F_{j-1}^\perp\).  Therefore
\[
 \bar r_j\in F_{j-1}^\perp
 =\operatorname{tan}(\bar a_{j-1},K),
 \qquad 1\leq j<s,
\]
and the terminal decomposition similarly gives
\[
 r_s\in\operatorname{tan}(\bar a_{s-1},K).
\]

If \(s=1\), then \(r_1\in F_0^\perp=K^\perp=\{0\}\), so setting
\(q_1=a_1\), \(z_1=0\), and \(y_1=g_1\) proves the result.  Suppose
therefore that \(s\geq2\).  Apply Corollary
\ref{cor:product-lifted-tangent}(b) first to
\(r_s\in\operatorname{tan}(\bar a_{s-1},K)\), with prescribed multiplier
one.  It supplies \(M_{s-1}>0\) and auxiliary variables certifying
\[
 z_s:=r_s\in\operatorname{tan}'(M_{s-1}\bar a_{s-1},K).
\]
Next, for \(j=s-1,s-2,\ldots,2\), apply the same corollary to
\(\bar r_j\in\operatorname{tan}(\bar a_{j-1},K)\), with prescribed
multiplier \(M_j\).  It supplies \(M_{j-1}>0\) and auxiliary variables
certifying
\[
 z_j:=M_j\bar r_j
 \in\operatorname{tan}'(M_{j-1}\bar a_{j-1},K).
\]
After completing this backward recursion, set
\[
 q_j=M_j\bar a_j,\qquad z_j=M_j\bar r_j
 \quad(1\leq j<s),
 \qquad q_s=a_s,\quad z_s=r_s,
\]
and define \(y_j=q_j+z_j\).  Since
\(\bar r_1\in F_0^\perp=K^\perp=\{0\}\), one has \(z_1=0\).  The displayed
certificates are therefore exactly the memberships in
\eqref{eq:adjacent-chain-relation}.  Finally,
\(\bar a_j+\bar r_j=\sum_{k=1}^j g_k\) for \(j<s\), while
\(a_s+r_s=g_s\), so \eqref{eq:facial-chain-encoding} follows.
\end{proof}

Figure~\ref{fig:three-stage-adjacent-chain} illustrates the construction in
Lemma~\ref{lem:adjacent-lifted-tangent-chain}(b) when \(s=3\).

\begin{figure}[htbp]
\centering
\begin{tikzpicture}[
 >=Latex,
 multiplier/.style={draw,rounded corners,minimum width=2.3cm,
   minimum height=7mm,font=\footnotesize},
 stage/.style={draw,rounded corners,align=center,text width=3.0cm,
   inner sep=3pt,font=\footnotesize},
 labeltext/.style={font=\scriptsize}
]
\node[font=\footnotesize\bfseries] at (0,0.8)
  {Multipliers chosen backward};
\node[multiplier] (m1) at (-4.3,0) {\(M_1\)};
\node[multiplier] (m2) at (0,0) {\(M_2\)};
\node[multiplier] (m3) at (4.3,0) {\(M_3=1\)};
\draw[->,thick] (m3) -- node[above,labeltext] {lift \(r_3\)} (m2);
\draw[->,thick] (m2) -- node[above,labeltext] {lift \(M_2\bar r_2\)} (m1);

\node[stage] (s1) at (-4.3,-2.55) {\textbf{Stage 1}\\
  \(y_1=M_1g_1\)\\ \(q_1=M_1\bar a_1\)\\ \(z_1=0\)};
\node[stage] (s2) at (0,-2.55) {\textbf{Stage 2}\\
  \(y_2=M_2(g_1+g_2)\)\\ \(q_2=M_2\bar a_2\)\\
  \(z_2=M_2\bar r_2\)};
\node[stage] (s3) at (4.3,-2.55) {\textbf{Stage 3}\\
  \(y_3=g_3\)\\ \(q_3=a_3\)\\ \(z_3=r_3\)};
\draw[->,dashed] (m1) -- node[right,labeltext] {defines} (s1);
\draw[->,dashed] (m2) -- node[right,labeltext] {defines} (s2);
\draw[->,dashed] (m3) -- node[right,labeltext] {defines} (s3);
\draw[thick] (s1.south east) to[bend right=17]
  node[below,labeltext] {\(z_2\in\operatorname{tan}'(q_1,K)\)}
  (s2.south west);
\draw[thick] (s2.south east) to[bend right=17]
  node[below,labeltext] {\(z_3\in\operatorname{tan}'(q_2,K)\)}
  (s3.south west);
\end{tikzpicture}
\caption{Three-stage instance of the construction in
Lemma~\ref{lem:adjacent-lifted-tangent-chain}(b).  Bars denote cumulative
cone and tangent parts for \(j<3\); the final box carries the unscaled terminal
slack.}
\label{fig:three-stage-adjacent-chain}
\end{figure}
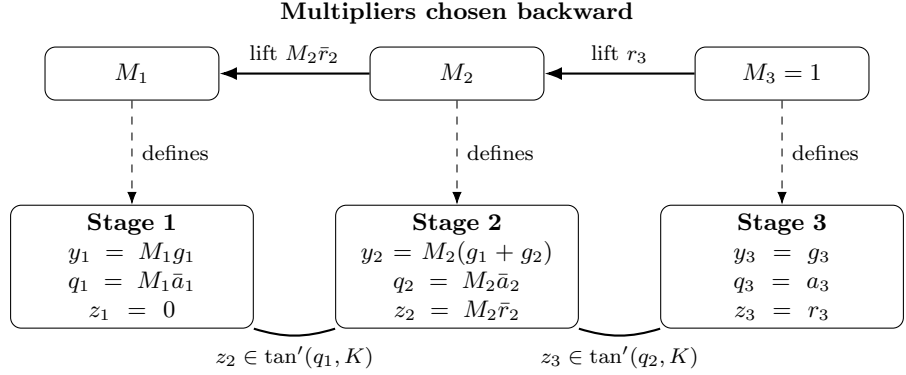

\subsection{An exact dual in SOCP form}
\label{sec:direct-exact-dual}

Combining the SOCP-representable mapping above with Pataki's extended-duality
result \cite[Corollary~3]{Pataki2013} yields the following explicit dual.  We
include a self-contained proof to verify the displayed formulation and its
polynomial size.

Put \(T=\rank K\).  Apply the constraints in
\eqref{eq:adjacent-chain-relation} to the program \((P)\) in
\eqref{eq:intro-primal} with \(s=T+1\).
The indices \(1,\ldots,T\) encode the cumulative reducing directions, while
index \(T+1\) carries the terminal dual slack.  With
\(\eta,\lambda_j\in\R^m\), define
\begin{equation}
\begin{aligned}
 (D_{\mathrm{SOCP}})\qquad
 \sup\quad &b^T\eta\\
 \text{subject to}\quad
 &y_j=A^T\lambda_j,\quad b^T\lambda_j=0
       &&(1\leq j\leq T),\\
 &y_{T+1}=c-A^T\eta,\\
 &y_j=q_j+z_j,\quad q_j\in K
       &&(1\leq j\leq T+1),\\
 &z_1=0,\quad z_j\in\operatorname{tan}'(q_{j-1},K)
       &&(2\leq j\leq T+1).
\end{aligned}
 \label{eq:direct-exact-dual}
\end{equation}
Each membership in the last line is expanded through a fresh
copy of
\eqref{eq:product-lift-source}--\eqref{eq:product-lift-image}; this explicit
lift, rather than the membership notation, makes the formulation an SOCP.
Its primary variables are
\(\eta,\lambda_j,q_j,w_{ji},v_{jiR}^+,v_{jiR}^-\); the symbols
\(y_j,z_j\) are affine abbreviations.  Free scalars may be split into
differences of nonnegative variables.

We measure scalar formulation size by the total number of scalar
decision-variable components, scalar affine equations, and coefficient
occurrences.  A variable in \(\mathcal L_d\) contributes \(d\) scalar
components, and each real coefficient is treated as one atomic scalar.

\begin{theorem}[Exact dual in SOCP form]
\label{thm:direct-exact-dual}
Every feasible point of \((D_{\mathrm{SOCP}})\) gives a lower bound on
\((P)\).  If \((P)\) is feasible and \(p^*>-\infty\), then
\((D_{\mathrm{SOCP}})\) attains \(p^*\), without a constraint qualification.
The scalar formulation size of \((D_{\mathrm{SOCP}})\) is
\begin{equation}
 O\!\left(T\sum_{i=1}^\nu n_i^3
   +T\,\operatorname{size}(A,b,c)\right),
 \label{eq:lifted-tangent-polynomial-size}
\end{equation}
where \(\operatorname{size}(A,b,c)\)
denotes the number of scalar coefficient occurrences of the input data in the
displayed primal formulation.
\end{theorem}

\begin{proof}
For weak duality, let a point of \((D_{\mathrm{SOCP}})\) be feasible.  If
\(x\) is primal feasible, then for \(j\leq T\),
\[
 \langle y_j,x\rangle
 =\langle A^T\lambda_j,x\rangle=b^T\lambda_j=0.
\]
Lemma~\ref{lem:adjacent-lifted-tangent-chain}(a), with \(s=T+1\), therefore
gives \(\langle y_{T+1},x\rangle\geq0\).  Hence
\[
 \langle c,x\rangle-b^T\eta
 =\langle c-A^T\eta,x\rangle
 =\langle y_{T+1},x\rangle\geq0.
\]

For attainment, Pataki's facial-reduction algorithm
\cite[Section~3, especially Theorem~1]{Pataki2013} supplies a sequence of
length \(\ell\leq T\), starting from \(F_0=K\), such that
\[
 g_j=A^T\theta_j\in F_{j-1}^*,
 \qquad b^T\theta_j=0,
 \qquad F_j=F_{j-1}\cap g_j^\perp
 \qquad(1\leq j\leq\ell),
\]
where \(F_\ell\) is the minimal face containing the primal feasible set.  To
fill the fixed \(T\) positions in \((D_{\mathrm{SOCP}})\), set
\(\theta_j=0\), \(g_j=0\), and \(F_j=F_\ell\) for \(\ell<j\leq T\).  This
padded list satisfies the displayed relations for \(1\leq j\leq T\), with
\(F_T=F_\ell\).  The feasible set meets \(\operatorname{ri}F_T\), so Slater
strong duality for the problem over the minimal face
\cite[p.~615]{Pataki2013} gives
\(\eta^*\) and the terminal slack
\[
 g_{T+1}=c-A^T\eta^*\in F_T^*,
 \qquad b^T\eta^*=p^*.
\]
Lemma~\ref{lem:adjacent-lifted-tangent-chain}(b), with \(s=T+1\), now
supplies variables satisfying \eqref{eq:adjacent-chain-relation} and positive
multipliers such that
\[
 y_j=M_j\sum_{k=1}^j g_k\quad(j\leq T),
 \qquad y_{T+1}=g_{T+1}.
\]
Taking \(\lambda_j=M_j\sum_{k=1}^j\theta_k\) and
\(\eta=\eta^*\) gives a feasible point of
\((D_{\mathrm{SOCP}})\) with value \(p^*\).

Finally, \(|\mathcal R_i|=1\) on a scalar block and
\(1+\binom{n_i}{2}\) on a Lorentz block.  Each of the \(T\) lifted tangent
constraints therefore uses
\(O(\sum_{i=1}^\nu n_i^3)\) scalar conic coordinates.  Adding the cone
vectors \(q_j\) and the affine equations containing the input data gives
\eqref{eq:lifted-tangent-polynomial-size}.
\end{proof}

\paragraph{Encoding consequence.}
The construction only copies the input coefficients and otherwise uses
\(0,1,-1\).  Consequently, for rational input data, the constructed
formulation also has polynomial binary encoding length.  This does not imply
polynomial-time solvability or a polynomial bound on the bit length of an
optimal rational certificate.

\begin{remark}[The Slater case]
If the primal meets \(\operatorname{ri}K\), no facial-reduction step is
needed, and the fixed positions in \((D_{\mathrm{SOCP}})\) may all be filled
with zero vectors.  The terminal slack condition is then simply
\(c-A^T\eta\in K\), so \((D_{\mathrm{SOCP}})\) reduces to the ordinary SOCP
dual.
\end{remark}

\subsection{Infeasibility as a consequence of exact duality}
\label{sec:infeasibility-certificates}

We now apply a standard homogenization consequence of exact duality.  Ramana's
SDP exact dual already yields a polynomial-size exact Farkas lemma
\cite{Ramana1997}, and Liu and Pataki use essentially the homogenized program
below \cite[proof of Theorem~4]{LiuPataki2018}.  The argument itself is
general.  What is specific here is that Theorem
\ref{thm:direct-exact-dual} makes the resulting certificate system a
polynomial-size SOCP.

This consequence is especially relevant under weak infeasibility, when
ordinary strict separation does not provide an attained normalized
certificate.  We begin with an affine SOCP feasibility system
\begin{equation}
 \Sigma=\{x\in\R^d:D_i x+d_i\in\Lcone_{n_i}\ (1\leq i\leq \nu),
 Ex=f\}.
 \label{eq:affine-socp}
\end{equation}
Write the free vector as \(x=x^+-x^-\), where
\(x^+,x^-\in\R_+^d\).  Introduce \(\tau\in\R_+\) and
\(s_i\in\Lcone_{n_i}\), and set
\begin{equation}
 h=(\tau,x^+,x^-,s_1,\ldots,s_\nu)\in C,
 \qquad
 C:=\R_+\times\R_+^d\times\R_+^d
       \times\prod_{i=1}^\nu\Lcone_{n_i}.
 \label{eq:homogeneous-product-cone}
\end{equation}
Write \(h_\tau=\tau\) for the first coordinate of \(h\).  The homogenized
affine equations are
\[
 s_i=D_i(x^+-x^-)+\tau d_i,\qquad
 E(x^+-x^-)=\tau f.
\]
The nonnegative scalar factors can be represented using Lorentz-cone
constraints, since \(t\geq0\) if and only if \((t,0)\in\Lcone_2\).
Let \(B\) collect the homogeneous equations and set \(L=\ker B\).  The
homogenized feasible cone is
\begin{equation}
 H_\Sigma=L\cap C.
 \label{eq:homogeneous-pair}
\end{equation}
Scaling gives the equivalence
\begin{equation}
 \Sigma\ne\varnothing
 \quad\Longleftrightarrow\quad
 H_\Sigma\text{ contains a point with }\tau>0.
 \label{eq:positive-tau-test}
\end{equation}
Thus \(\Sigma\) is infeasible exactly when \(\tau=0\) throughout
\(H_\Sigma\).

Consider the always-feasible homogeneous program
\begin{equation}
 (P_\Sigma)\qquad
 \inf\{-h_\tau:Bh=0,\ h\in C\},
 \label{eq:homogeneous-certificate-primal}
\end{equation}
Let \(e_\tau\) be the coordinate vector corresponding to \(\tau\).  Apply
\((D_{\mathrm{SOCP}})\) to \((P_\Sigma)\) by taking
\(A=B\), \(b=0\), and \(c=-e_\tau\), and denote the resulting affine SOCP
feasibility system by \(\mathfrak N(\Sigma)\).

\begin{corollary}[Infeasibility within SOCP]
\label{cor:pure-socp-infeasibility-certificate}
The transformation above constructs a polynomial-size SOCP system
\(\mathfrak N(\Sigma)\) that is feasible if and only if
\begin{equation}
 h_\tau=0\qquad\text{for every }h\in H_\Sigma,
 \label{eq:no-positive-tau}
\end{equation}
equivalently, if and only if \(\Sigma\) is infeasible.  The transformation
applies to arbitrary real input data and uses only fixed rational constants
in addition to those data.
\end{corollary}

\begin{proof}
If \eqref{eq:no-positive-tau} holds, then
\((P_\Sigma)\) has the finite value zero.  Theorem
\ref{thm:direct-exact-dual} gives a feasible dual attaining zero.  If instead
some \(h\in H_\Sigma\) has \(h_\tau>0\), conic scaling makes
\((P_\Sigma)\) unbounded below.  Any feasible point of its dual would have
objective \(b^T\eta=0\) and, by weak duality, would give the impossible lower
bound \(-h_\tau\geq0\) for every \(h\in H_\Sigma\).  Thus the dual is
infeasible.  The equivalence with infeasibility of \(\Sigma\) is
\eqref{eq:positive-tau-test}.  Polynomial size and the coefficient claim
follow from Theorem~\ref{thm:direct-exact-dual} and the explicit
homogenization above.
\end{proof}

\begin{example}[A weakly infeasible SOCP]
\label{ex:worked-example}
Consider the SOCP feasibility system in the coordinates
\((x_1,x_2,x_3)\in\R^3\):
\begin{equation}
 (x_1,x_2,x_3)\in\Lcone_3,\qquad x_1=x_2,\qquad x_3=1.
 \label{eq:weak-example-affine}
\end{equation}
The two equalities define the affine line
\[
 \mathcal A=\{(x_1,x_2,x_3):x_1=x_2,\ x_3=1\}
     =\{(t,t,1):t\in\R\}.
\]
The system is infeasible because membership in \(\Lcone_3\) would require
\(t\geq\sqrt{t^2+1}\).  Nevertheless,
\[
 \operatorname{dist}\bigl((t,t,1),\Lcone_3\bigr)
 =\frac{\sqrt{t^2+1}-t}{\sqrt{2}}\longrightarrow0
 \qquad(t\to+\infty).
\]
Hence \(\mathcal A\) and \(\Lcone_3\) cannot be strictly separated, as
illustrated in Figure~\ref{fig:weak-infeasibility-geometry}, and the ordinary
conic alternative provides no infeasibility certificate.

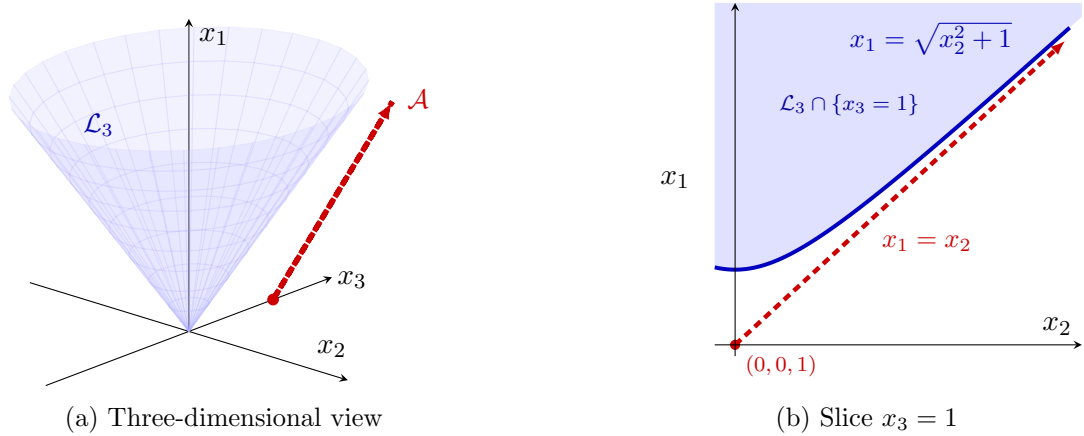
\begin{figure}[H]
\centering
\begin{minipage}[b]{0.58\textwidth}
\centering
\begin{tikzpicture}
\begin{axis}[
 width=\linewidth,
 height=0.88\linewidth,
 view={42}{25},
 axis lines=center,
 xlabel={\(x_2\)},
 ylabel={\(x_3\)},
 zlabel={\(x_1\)},
 ylabel style={xshift=25pt,yshift=-5pt},
 xmin=-2.0,xmax=2.0,
 ymin=-1.7,ymax=1.7,
 zmin=0,zmax=2.0,
 ticks=none,
 clip=false
]
\addplot3[
 surf,
 shader=flat,
 draw=blue!32,
 fill=blue!25,
 opacity=0.20,
 domain=0:1.55,
 y domain=0:360,
 samples=10,
 samples y=33
] ({x*cos(y)},{x*sin(y)},{x});

\addplot3[
 red!80!black,
 densely dashed,
 line width=2pt,
 -{Latex[length=2.3mm]},
 domain=0:1.5,
 samples=2
] ({x},{1},{x});

\addplot3[
 red!80!black,
 only marks,
 mark=*,
 mark size=2pt
] coordinates {(0,1,0)};

\node[red!80!black,anchor=south east,font=\small,
      xshift=25pt,yshift=6pt]
  at (axis cs:1.25,1,1.25) {\(\mathcal A\)};
\node[blue!70!black,anchor=south east,font=\small]
  at (axis cs:-0.55,-0.25,1.15) {\(\Lcone_3\)};
\end{axis}
\end{tikzpicture}
\par\smallskip
\small (a) Three-dimensional view
\end{minipage}\hfill
\begin{minipage}[b]{0.39\textwidth}
\centering
\begin{tikzpicture}
\begin{axis}[
 width=\linewidth,
 height=0.97\linewidth,
 axis lines=middle,
 axis on top,
 xlabel={\(x_2\)},
 xmin=-0.25,xmax=4.25,
 ymin=-0.15,ymax=4.55,
 ticks=none,
 clip=false
]
\node[anchor=east] at (axis description cs:-0.04,0.5) {\(x_1\)};
\path[name path=slicetop]
  (axis cs:-0.25,4.55) -- (axis cs:4.25,4.55);
\addplot[
 name path=sliceboundary,
 draw=none,
 domain=-0.25:4.25,
 samples=80
] {sqrt(x^2+1)};
\addplot[blue!12]
  fill between[of=sliceboundary and slicetop];
\addplot[
 blue!75!black,
 solid,
 line width=1.5pt,
 domain=-0.25:4.1,
 samples=80
] {sqrt(x^2+1)};
\addplot[
 red!80!black,
 densely dashed,
 line width=1.7pt,
 -{Latex[length=2mm]},
 domain=0:4.05,
 samples=2
] {x};
\addplot[
 red!80!black,
 only marks,
 mark=*,
 mark size=1.7pt
] coordinates {(0,0)};
\node[red!80!black,anchor=north west,font=\scriptsize]
  at (axis cs:0,0) {\((0,0,1)\)};
\node[red!80!black,anchor=center,font=\small]
  at (axis cs:2.35,1.35) {\(x_1=x_2\)};
\node[blue!70!black,anchor=center,font=\small]
  at (axis cs:2.40,4.08) {\(x_1=\sqrt{x_2^2+1}\)};
\node[blue!60!black,align=center,font=\scriptsize]
  at (axis cs:1.40,3.20)
  {\(\Lcone_3\cap\{x_3=1\}\)};
\end{axis}
\end{tikzpicture}
\par\smallskip
\small (b) Slice \(x_3=1\)
\end{minipage}
\caption{Geometry of the weakly infeasible system
\eqref{eq:weak-example-affine}.  In the slice \(x_3=1\), the boundary of
\(\Lcone_3\) is the solid curve \(x_1=\sqrt{x_2^2+1}\).  The dashed affine
line \(\mathcal A\) satisfies \(x_1=x_2\), so it lies strictly below this
boundary, while the distance between the two curves tends to zero as
\(x_2\to+\infty\).}
\label{fig:weak-infeasibility-geometry}
\end{figure}

Homogenizing with \(\tau\geq0\) gives, in the coordinate order
\((\tau;x_1,x_2,x_3)\),
\begin{equation}
 C=\R_+\times\Lcone_3,
 \qquad
 B=\begin{pmatrix}
     0&1&-1&0\\
    -1&0& 0&1
   \end{pmatrix}.
 \label{eq:weak-example-B}
\end{equation}
Indeed,
\[
 \ker B\cap C=\{(0;t,t,0):t\geq0\},
\]
so \(\tau\) vanishes although the homogeneous feasible cone is nonzero.
\begin{samepage}
We now exhibit an infeasibility certificate.  Since \(\rank C=3\), the system
\(\mathfrak N(\Sigma)\) has three positions for facial-reduction information
and a fourth position for the terminal slack.  Specializing
\eqref{eq:direct-exact-dual} gives
\begin{subequations}
\label{eq:weak-example-exact-dual}
\begin{align}
 y_j&=B^T\lambda_j &&(1\leq j\leq3),
 \label{eq:weak-example-exact-dual-range}\\
 y_4&=-e_\tau-B^T\eta,
 \label{eq:weak-example-exact-dual-terminal}\\
 y_j&=q_j+z_j,\qquad q_j\in C &&(1\leq j\leq4),
 \label{eq:weak-example-exact-dual-decomposition}\\
 z_1&=0,\qquad
 z_j\in\operatorname{tan}'(q_{j-1},C) &&(2\leq j\leq4).
 \label{eq:weak-example-exact-dual-tangent}
\end{align}
\end{subequations}
\end{samepage}
where \(\eta,\lambda_j\in\R^2\).  A feasible point is
\begin{equation}
 \eta=0,\qquad
 \lambda_1=(1,0),\qquad
 \lambda_2=\lambda_3=(1,-1),
 \label{eq:weak-example-multipliers}
\end{equation}
together with
\begin{equation}
\begin{array}{c|ccc}
 j&y_j&q_j&z_j\\
 \hline
 1&(0;1,-1,0)&(0;1,-1,0)&0\\
 2&(1;1,-1,-1)&(1;1,-1,0)&(0;0,0,-1)\\
 3&(1;1,-1,-1)&(1;1,-1,0)&(0;0,0,-1)\\
 4&(-1;0,0,0)&0&(-1;0,0,0).
\end{array}
 \label{eq:weak-example-feasible-point}
\end{equation}
Direct substitution using \eqref{eq:weak-example-B}, together with the table,
verifies that
\eqref{eq:weak-example-multipliers}--%
\eqref{eq:weak-example-feasible-point} satisfy
\eqref{eq:weak-example-exact-dual-range}--%
\eqref{eq:weak-example-exact-dual-decomposition}.
The three memberships in
\eqref{eq:weak-example-exact-dual-tangent} follow by direct substitution in
the blockwise representation
\eqref{eq:product-lift-source}--\eqref{eq:product-lift-image}.
Thus the displayed data extend to a feasible point of
\(\mathfrak N(\Sigma)\), certifying the infeasibility of \(\Sigma\).
\end{example}

\section{Relation to existing exact-duality frameworks}
\label{sec:related}

Ramana's exact dual for semidefinite programming augments the ordinary dual
slack with variables that encode enough of a facial-reduction sequence to
recover strong duality \cite{Ramana1997}.  Ramana, Tun\c{c}el, and Wolkowicz
made its connection with facial reduction explicit
\cite{RamanaTuncelWolkowicz1997}.  Pataki recast this construction using pairs
\((u_j,v_j)\) satisfying
\[
 u_j\in K^*,\qquad
 v_j\in\operatorname{tan}'(u_{j-1},K^*),
\]
with \(u_j+v_j\) carrying the facial-reduction information
\cite[Corollary~3]{Pataki2013}.

In the present self-dual SOCP setting, \(q_j\), \(z_j\), and
\(y_j=q_j+z_j\) play the respective roles of \(u_j\), \(v_j\), and
\(u_j+v_j\).  Proposition~\ref{prop:lifted-tangent} and
Corollary~\ref{cor:product-lifted-tangent} supply the Lorentz-cone
representation of \(\operatorname{tan}'\).  Thus the connection to Ramana's
dual is through Pataki's general derivation: the facial-reduction framework
comes from his work, while the SOCP representation of
\(\operatorname{tan}'\) is specific to the present construction.

This distinction explains how the present construction avoids the obstacle
encountered in the earlier Lorentz-cone specialization.  P\'olik and Terlaky
represent information orthogonal to the current face by a Lorentz--Siegel
cone, which is not a Lorentz cone \cite{PolikTerlaky2007}.  The formulation
here represents neither that cone nor the full tangent-space constraint.  It
uses the auxiliary mapping from
Corollary~\ref{cor:product-lifted-tangent} only along the facial-reduction
sequence.  Thus their particular extended dual leaves the SOCP class, whereas
the present lifted formulation does not.

Beyond this correspondence, three distinctions locate the result within the
existing literature.
\begin{enumerate}[leftmargin=2.2em]
\item Liu and Pataki's facial-reduction sequence cones provide a related
exact-duality framework \cite{LiuPataki2018}.  The result here instead
gives one affine formulation using only the original Lorentz-cone class.

\item Each auxiliary problem in a sequential facial-reduction method may
itself be an SOCP
\cite{WakiMuramatsu2013,LourencoMuramatsuTsuchiya2016}.  Such a method forms
later problems after solving earlier ones.  Here one polynomial-size affine
SOCP is constructed before any solving takes place.

\item Naldi and Sinn already prove polynomial-length feasibility and
infeasibility certificates for nice cones with arbitrary real coefficients
\cite{NaldiSinn2021}.  In the Blum--Shub--Smale (BSS) model, their result
places SOCP feasibility in
\(\mathsf{NP}_{\R}\cap\mathsf{coNP}_{\R}\).  In this model, real numbers are
exact inputs, and real arithmetic and comparisons have unit cost
\cite{BlumShubSmale1989}; this is different from the Turing bit model and does
not bound the bit length of a rational certificate.  Our claim is different:
an infeasibility certificate is itself a feasible point of one affine SOCP
feasibility system using only Lorentz-cone constraints.
\end{enumerate}

Accordingly, the proved contribution is a uniform polynomial-size
transformation from an affine SOCP feasibility system to an SOCP system
that is feasible exactly when the input system is infeasible, together with
the resulting attained exact dual over arbitrary real coefficients.  It is not
a general complexity result for exact SOCP solution and makes no priority
claim beyond this precise mathematical statement.

\section{Conclusion}

Every second-order cone program has a polynomial-size exact dual expressed
solely with affine equations and Lorentz-cone memberships.  The dual gives
valid bounds unconditionally and attains the primal value whenever the primal
is feasible with finite value.  The same construction transforms any affine
second-order cone feasibility system into another such system whose
feasibility is equivalent to infeasibility of the original, including in the
weakly infeasible case.

The construction therefore keeps both exact duality and infeasibility
certification within the second-order cone modeling class.  It achieves this
with an auxiliary mapping satisfying containment and rescaling along one
facial-reduction sequence, rather than a representation of the full
tangent-space constraint.  The
result concerns exact formulation size over arbitrary real coefficients; it
does not imply favorable numerical conditioning, polynomial-time exact
solution, or polynomial bit length of rational certificates.

\section*{Statements and Declarations}

\subsection*{Funding}
This work was supported in part by the Air Force Office of Scientific Research
under Award No.~FA9550-23-1-0508.

\subsection*{Competing interests}
The author has no relevant financial or non-financial interests to disclose.

\subsection*{Data availability}
No datasets were generated or analyzed during the current study.


\begin{thebibliography}{99}
\small

\bibitem{BlumShubSmale1989}
L.~Blum, M.~Shub, and S.~Smale,
\emph{On a theory of computation and complexity over the real numbers:
NP-completeness, recursive functions and universal machines},
Bulletin of the American Mathematical Society 21 (1989), 1--46.
\url{https://doi.org/10.1090/S0273-0979-1989-15750-9}.

\bibitem{BorweinWolkowicz1981Facial}
J.~M.~Borwein and H.~Wolkowicz,
\emph{Facial reduction for a cone-convex programming problem},
Journal of the Australian Mathematical Society, Series A 30(3) (1981),
369--380.
\url{https://doi.org/10.1017/S1446788700017250}.

\bibitem{BorweinWolkowicz1981Regularizing}
J.~M.~Borwein and H.~Wolkowicz,
\emph{Regularizing the abstract convex program},
Journal of Mathematical Analysis and Applications 83(2) (1981), 495--530.
\url{https://doi.org/10.1016/0022-247X(81)90138-4}.

\bibitem{DrusvyatskiyWolkowicz2017}
D.~Drusvyatskiy and H.~Wolkowicz,
\emph{The many faces of degeneracy in conic optimization},
Foundations and Trends in Optimization 3(2) (2017), 77--170.
\url{https://doi.org/10.1561/2400000011}.

\bibitem{LiuPataki2018}
M.~Liu and G.~Pataki,
\emph{Exact duals and short certificates of infeasibility and weak
infeasibility in conic linear programming},
Mathematical Programming 167 (2018), 435--480.
\url{https://doi.org/10.1007/s10107-017-1136-5}.

\bibitem{LourencoMuramatsuTsuchiya2016}
B.~F.~Lourenço, M.~Muramatsu, and T.~Tsuchiya,
\emph{Weak infeasibility in second order cone programming},
Optimization Letters 10 (2016), 1743--1755.
\url{https://doi.org/10.1007/s11590-015-0982-4}.

\bibitem{LourencoMuramatsuTsuchiya2021}
B.~F.~Lourenço, M.~Muramatsu, and T.~Tsuchiya,
\emph{Solving SDP completely with an interior point oracle},
Optimization Methods and Software 36(2--3) (2021), 425--471.
\url{https://doi.org/10.1080/10556788.2020.1850720}.

\bibitem{NaldiSinn2021}
S.~Naldi and R.~Sinn,
\emph{Conic programming: infeasibility certificates and projective geometry},
Journal of Pure and Applied Algebra 225(7) (2021), 106605.
\url{https://doi.org/10.1016/j.jpaa.2020.106605}.

\bibitem{Pataki2013}
G.~Pataki,
\emph{Strong duality in conic linear programming: facial reduction and
extended duals},
in Computational and Analytical Mathematics, Springer Proceedings in
Mathematics \& Statistics 50 (2013), 613--634.
\url{https://doi.org/10.1007/978-1-4614-7621-4_28}.

\bibitem{PolikTerlaky2007}
I.~P\'olik and T.~Terlaky,
\emph{Exact duality for optimization over symmetric cones},
AdvOL Report 2007/10 (2007).
\url{https://optimization-online.org/wp-content/uploads/2007/08/1754.pdf}.

\bibitem{Ramana1997}
M.~V.~Ramana,
\emph{An exact duality theory for semidefinite programming and its complexity
implications}, Mathematical Programming 77 (1997), 129--162.
\url{https://doi.org/10.1007/BF02614433}.

\bibitem{RamanaTuncelWolkowicz1997}
M.~V.~Ramana, L.~Tun\c{c}el, and H.~Wolkowicz,
\emph{Strong duality for semidefinite programming},
SIAM Journal on Optimization 7(3) (1997), 641--662.
\url{https://doi.org/10.1137/S1052623495288350}.

\bibitem{WakiMuramatsu2013}
H.~Waki and M.~Muramatsu,
\emph{Facial reduction algorithms for conic optimization problems},
Journal of Optimization Theory and Applications 158 (2013), 188--215.
\url{https://doi.org/10.1007/s10957-012-0219-y}.

\end{thebibliography}
\end{document}